\documentclass[11pt,letterpaper]{amsart}
\usepackage[T1]{fontenc}
\usepackage{lmodern}
\usepackage{amsmath,amssymb,amsthm}
\usepackage{microtype}
\usepackage[hidelinks]{hyperref}
\hypersetup{pdftitle={Deformation invariance of canonical nefness in smooth Kahler morphisms},
 pdfauthor={Mu-Lin Li, Xiao-Lei Liu and Sheng Rao}}
\numberwithin{equation}{section}
\allowdisplaybreaks[1]
\newtheorem{theorem}{Theorem}[section]
\newtheorem{proposition}[theorem]{Proposition}
\newtheorem{lemma}[theorem]{Lemma}
\newtheorem{corollary}[theorem]{Corollary}
\theoremstyle{definition}
\newtheorem{definition}[theorem]{Definition}
\theoremstyle{remark}
\newtheorem{remark}[theorem]{Remark}

\newcommand{\R}{\mathbb R}
\newcommand{\C}{\mathbb C}
\newcommand{\Z}{\mathbb Z}
\newcommand{\Q}{\mathbb Q}
\newcommand{\Pp}{\mathbb P}
\newcommand{\OO}{\mathcal O}
\newcommand{\KK}{\mathcal K}
\newcommand{\BC}{\mathrm{BC}}
\newcommand{\ddc}{dd^{c}}
\newcommand{\PD}{\operatorname{PD}}
\newcommand{\Hom}{\operatorname{Hom}}
\newcommand{\Aut}{\operatorname{Aut}}
\newcommand{\PGL}{\operatorname{PGL}}
\newcommand{\Locus}{\operatorname{Locus}}
\newcommand{\Nef}{\operatorname{Nef}}
\newcommand{\red}{\mathrm{red}}
\newcommand{\rk}{\operatorname{rank}}

\newcommand{\Mspace}{\mathcal P}
\newcommand{\Cyc}{\mathcal B}
\newcommand{\arxiv}[1]{\href{https://arxiv.org/abs/#1}{arXiv:#1}}

\newcommand{\X}{\mathcal X}
\newcommand{\NefLocus}{\mathcal N}

\newcommand{\ddbar}{\partial\bar\partial}
\newcommand{\source}[2][]{\if\relax\detokenize{#1}\relax[#2]\else[#2, #1]\fi}
\numberwithin{equation}{section}

\title[Deformation invariance of canonical nefness]{\bfseries Deformation invariance of canonical nefness
 in smooth K\"ahler morphisms}
\author{Mu-Lin Li~~~~ Xiao-Lei Liu ~~~~Sheng Rao}

\thanks{The first author is supported by NSFC (No. 12271073 and 12271412), the second author is supported by NSFC (No. 12271073) and the third author is supported by NSFC ( No. 12271412, W2441003, 12671107) and Hubei Provincial
Innovation Research Group Project (No. 2025AFA044).}

\date{}

\begin{document}
\maketitle
\vspace{-1.2em}
\begin{abstract}
Let $f\colon X\to\Delta$ be a smooth K\"ahler morphism from complex manifold $X$ to the unit disc. We prove that the canonical
bundle $K_{X_t}$ is nef for \emph{every} fiber as soon as it is nef
for \emph{one} fiber.
This answers, in arbitrary dimension and in the K\"ahler setting, the
deformation-openness problem for non-nefness of the canonical bundle
raised by Campana and Peternell.
\end{abstract}

\section{Introduction}\label{sec:introduction}

Let $\Delta=\{t\in\C:|t|<1\}$, and let $f\colon X\to\Delta$ be a
smooth proper holomorphic map of complex manifolds; write
$X_t=f^{-1}(t)$ for the fiber. We study a stability question for the
canonical bundle: \emph{is the nefness of $K_{X_t}$ constant along the
family?} In particular, if the canonical bundle of one fiber is nef,
must the canonical bundle of every fiber be nef? The reverse direction
--- that non-nefness persists --- is what one needs to run a Mori
program fibrewise over the disc, and it is this direction that turns
out to be difficult.

The question goes back to Campana and Peternell, who isolate it as
``the first basic question for a Mori theory in the K\"ahler case''
\cite[\S3.13]{CP99}. They  formulate the question as the following openness property.

\medskip
\noindent\emph{Problem} (Campana--Peternell \cite[Problem~3.14]{CP99}).
\emph{Let $X\to\Delta$ be a family of compact K\"ahler manifolds.
Assume that $K_{X_0}$ is not nef. Is then $K_{X_t}$ not nef for all
(small) $t$?}

It is noted that the deformation openness of  non-nefness in the analytic case is really interesting and much
more delicate in \cite{AP97}.


\begin{definition}[K\"ahler morphism]\label{def:kahler-morphism}
	A smooth proper morphism $f\colon\X\to\Delta$ is a \emph{K\"ahler morphism} if there are an open covering $\{U_i\}$ of $\X$ and smooth real-valued functions $\varphi_i$ on $U_i$ such that $\varphi_i-\varphi_j$ is the real part of a holomorphic function on $U_i\cap U_j$ and $\sqrt{-1}\,\ddbar\varphi_i$ is positive on $T_{\X/\Delta}|_{U_i}$.
\end{definition}
The local forms glue to a closed real $(1,1)$-form whose restriction to
each fiber is K\"ahler. We establish the stronger conclusion that canonical nefness is constant for K\"ahler morphism, with no restriction on dimension.
\begin{theorem}\label{thm:main-2}
	Let $f\colon\X\to\Delta$ be a smooth K\"ahler morphism. If $K_{X_{t_*}}$ is nef for some $t_*\in\Delta$, then $K_{X_t}$ is nef for every $t\in\Delta$. Equivalently,
	\begin{equation}\label{eq:nef-locus-2}
		\NefLocus(f):=\{t\in\Delta:K_{X_t}\text{ is nef}\}\quad\text{is either }\varnothing\text{ or }\Delta.
	\end{equation}
\end{theorem}

Wi\'sniewski \source{17, 18} established deformation invariance of nef values and prove this openness property in the projective setting. For smooth projective families over a complex base, Andreatta--Peternell \source{1} proved persistence of canonical non-nefness when the special fiber has an extremal contraction that is not small; in particular, their result applies in dimensions at most four. Li--Rao--Wang \source{14} proved deformation invariance of nefness for adjoint canonical line bundles in smooth projective families and treated canonical nefness in smooth K\"ahler families of threefolds. Theorem~\ref{thm:main-2} addresses the nonprojective K\"ahler case in arbitrary dimension, without hypotheses on extremal contractions.

All of these results are either projective or bounded in dimension (or
both), and the treatment of $K_X$ itself in \cite{AP97} is further
subject to the small-contraction.

Our contribution removes every one of these restrictions at once:
we prove the result for smooth K\"ahler families of \emph{arbitrary}
relative dimension, with no hypothesis on the type of extremal contraction. We in fact establish the a priori
stronger \emph{constancy} statement, whereas Campana--Peternell only
ask that non-nefness survive small deformations. To our knowledge this
is the first result in this direction that is simultaneously valid in
the K\"ahler category, in every dimension, and without any proviso on
the extremal contractions.

 From Fujiki
\source[Lemma 4.4(2)]{6}, the inverse image $f^{-1}(G)$ is K\"ahler for every relatively compact subdomain
$G\Subset\Delta$. So Theorem~\ref{thm:main-2} reduces  to  the following special case.
\begin{theorem}\label{thm:main}
	Let $f\colon X\to\Delta$ be a smooth proper surjective holomorphic map from a K\"ahler manifold to the unit disc. If $K_{X_{t_*}}$ is nef for some $t_*\in\Delta$, then $K_{X_t}$ is nef for every $t\in\Delta$. Equivalently,
	\begin{equation}\label{eq:nef-locus}
		\NefLocus(f):=\{t\in\Delta:K_{X_t}\text{ is nef}\}\quad\text{is either }\varnothing\text{ or }\Delta.
	\end{equation}
\end{theorem}

\medskip
\noindent\textbf{Ideas of the proof of Theorem~\ref{thm:main}.}
The main idea is to compare a Lefschetz bound for a nef-threshold
contraction with the deformations of its minimal rational curves.
Building on the nef-value strategy of Wi\'sniewski \cite{Wi91,Wi98},
we carry out this comparison without a relative polarization and using
nefness on only one fiber. The main difficulties, and the constructions
that overcome them, are the following.

\medskip
\noindent\textbf{(1) A Lefschetz constraint for a transcendental boundary class.}
Suppose that $M=X_0$ has non-nef canonical bundle while
$K_{X_{t_*}}$ is nef, and put $n=\dim M$.
A K\"ahler form $\omega$ on $X$ gives a positive threshold $\tau$
for which $\alpha_0=c_1(K_M)+\tau[\omega_0]$ is nef but not
K\"ahler, where $\omega_0=\omega|_M$. The difficulty is that hard
Lefschetz need not hold for a nef boundary class. We overcome this by
observing that the \emph{same flat cohomology class} is K\"ahler on
$X_{t_*}$. The classical theorem \cite{GH78} on that fiber therefore
gives the Lefschetz isomorphisms on $M$; no limiting argument is involved.

The transcendental base-point-free theorem of Hacon--Xie
\cite[Theorem~1.4]{HX26} realizes $\alpha_0$ as $g^*\beta$ in Bott--Chern cohomology for a
nontrivial projective contraction $g\colon M\to Y$ and a K\"ahler
class $\beta$ on $Y$. Combining the Lefschetz isomorphisms with a
rank-vanishing argument for the pullback of a K\"ahler form
representing $\beta$ gives the semismallness bound
$2\dim Z-\dim g(Z)\le n$ for every irreducible compact analytic
$Z\subset M$, even when $Y$ is singular
(Section~\ref{sec:threshold}).

\medskip
\noindent\textbf{(2) A compact analytic family of minimal rational curves.}
To apply bend-and-break, we need a compact parameter space with a
universal family of normalized rational curves, not merely a compact
space of image cycles. Barlet compactness alone does not control the
quotient by the noncompact reparametrization group.

For $g$-contracted rational curves of minimal anticanonical degree
$\ell>0$, the identity $\tau\int_C\omega_0=\ell$ fixes their energy.
Gromov compactness \cite{MS12}, together with minimality, excludes
bubbling and multiple covers. We also establish properness of the free
$\PGL_2(\C)$-action, obtaining a compact analytic moduli space
$\Mspace$ with a universal $\Pp^1$-family
(Appendix~\ref{app:moduli}). Neither $M$ nor $\Mspace$ is assumed
projective. Instead, we prove that the fibers of $\Mspace\to Y$ are
projective, which suffices to localize bend-and-break to the contraction
fibers (Section~\ref{sec:contraction}).

\medskip
\noindent\textbf{(3) Confinement from one nef fiber and the extra dimension.}
The remaining difficulty is to retain the deformation estimate in the
total space $X$ while forcing the curves to remain in $M$. Merely
observing that a negative curve cannot reach $X_{t_*}$ is insufficient:
we must confine its \emph{entire reduced local map germ}.

We use the relative Barlet space. Constant K\"ahler area makes each
irreducible component proper over $\Delta$, with image either a point
or all of $\Delta$. Constancy of the $K_{X/\Delta}$-degree and nefness
on $X_{t_*}$ exclude the second possibility for negative components.
Passing from maps to their image cycles therefore confines the entire
reduced local map germ of a minimal contracted curve to $M$
(Section~\ref{sec:barlet}).

Crucially, we estimate the local dimension of maps into $X$
\emph{before} invoking confinement. The extra dimension of $X$ is
retained after confinement and quotienting by reparametrizations,
giving a component $V\subset\Mspace$ with
$r=\dim V\ge n+\ell-2$ (Section~\ref{sec:deformation}).
For $Z=\Locus(V)$, $d=\dim Z$ and $e=\dim g(Z)$, bend-and-break
and the Lefschetz constraint now give
\[
 n+\ell\le r+2\le 2d-e\le n,
\]
contradicting $\ell>0$. This argument requires neither deforming the
contraction nor classifying its exceptional locus.

\medskip
\noindent\emph{Organization.}
Section~\ref{sec:preliminaries} fixes notation for Bott--Chern
cohomology and for the relative Barlet space, and
develops the confinement of negative curves, the key new tool.
Section~\ref{sec:threshold} constructs the nef threshold class and
records its hard Lefschetz property. Section~\ref{sec:contraction}
produces the projective contraction and the moduli space of minimal
rational curves (constructed in Appendix~\ref{app:moduli}).
Section~\ref{sec:deformation} carries out the deformation theory, and
Section~\ref{sec:proof} assembles the final contradiction.

{\bf Acknowledgement:} The authors thank Lingyao Xie for discussions on the results of their paper. AI plays a supporting role in this project. We used the  OpenAI GPT-6.0-pro model to polish the article and the proof of the appendix. The authors are fully responsible for all assertions in this paper.

\section{Preliminaries}\label{sec:preliminaries}

Unless stated otherwise, we work with the family in
Theorem~\ref{thm:main}. We may assume its fibers are connected and
write $n=\dim X_t$.

\subsection{Cohomology and positivity}
On a smooth complex manifold $M$, write
\[
 H^{1,1}_{\BC}(M,\R)
 =\frac{\{\text{real, smooth, $d$-closed $(1,1)$-forms on $M$}\}}
 {\{\ddc u:u\in C^\infty(M,\R)\}}.
\]
The normalization of $d^c$ will not matter. When $M$ is compact
K\"ahler, the $\partial\bar\partial$-lemma identifies this group with
the real $(1,1)$-part of de Rham cohomology. We use the same symbol for a
Bott--Chern class and its de Rham image when there is no ambiguity.
The K\"ahler cone is denoted by $\KK(M)$, and
\[
 \Nef(M)=\overline{\KK(M)}.
\]
In particular, the sum of a nef class and a K\"ahler class is
K\"ahler.

On a reduced K\"ahler space, a K\"ahler form is understood in the
usual local-potential sense: in local embeddings it is represented by
smooth strictly plurisubharmonic potentials, with pluriharmonic
differences.

\subsection{Relative compact cycles and their periods}
Write $\Cyc_1(X)$ for the reduced Barlet space of compact effective
one-cycles on $X$. A cycle is a finite sum
\[
 Z=\sum_{i=1}^q m_iC_i,
 \qquad m_i\in\Z_{>0},
\]
where the $C_i$ are distinct irreducible compact analytic curves.
We write $|Z|=\bigcup_i C_i$ for its support. The zero cycle will be
excluded throughout the relative argument.

The relative Barlet space is the reduced complex space
\begin{equation}\label{eq:relative-barlet}
 \Cyc_1(X/\Delta)
 =\{(t,Z)\in\Delta\times\Cyc_1(X):Z\ne0,\ |Z|\subset X_t\}.
\end{equation}
It has the usual analytic structure representing analytic families of
compact cycles contained in fibers, a holomorphic projection
$p\colon\Cyc_1(X/\Delta)\to\Delta$, and a holomorphic forgetful map to
$\Cyc_1(X)$.

\begin{lemma}\label{lem:periods}
Let $\theta$ be a smooth real $d$-closed two-form on $X$. The function
\[
 Z\longmapsto\int_Z\theta
\]
is locally constant on $\Cyc_1(X)$. In particular, the $\omega$-area
and the degree with respect to a fixed holomorphic line bundle on $X$
are constant on every connected family of compact one-cycles.
\end{lemma}

\begin{proof}
The homology class of a compact cycle is locally constant in the
cycle topology: sufficiently close cycles represent the same class
in a neighborhood of the reference support. Integration of a closed
form therefore has the same value on them. This standard topological
property of cycle spaces also gives the local constancy of K\"ahler
volume; see \cite[Theorem~1 and the following remark]{Bar99}.
For a line bundle, one may alternatively use a smooth Chern form:
its integral is a continuous integer-valued function, and hence is
locally constant. Restriction to the relative cycle space preserves
these conclusions.
\end{proof}

\begin{proposition}\cite[Proposition 3.9]{LL24}
\label{prop:barlet-proper}
Let $f\colon X\to\Delta$ be proper, and suppose that $X$ carries a
K\"ahler form $\omega$. For every irreducible component $B$ of
$\Cyc_1(X/\Delta)$, the map
\[
 p_B:=p|_B\colon B\longrightarrow\Delta
\]
is proper. Its image is either a single point or all of $\Delta$.
\end{proposition}

\begin{proof}
By Lemma~\ref{lem:periods}, the positive number
\[
 v_B=\int_Z\omega
\]
is independent of $(t,Z)\in B$. Fix a compact set $K\subset\Delta$.
Every cycle parametrized by $p_B^{-1}(K)$ is supported in the compact
set $f^{-1}(K)$ and has $\omega$-area $v_B$. The compactness theorem
for cycles \cite[Theorem~1]{Bar99}, which follows from Bishop's
theorem \cite{Bishop}, gives relative compactness of these cycles in
$\Cyc_1(X)$.

For completeness, consider a sequence $(t_j,Z_j)\in p_B^{-1}(K)$.
After passing to a subsequence, $t_j\to t\in K$ and $Z_j\to Z$ in
the cycle space. Continuity of integration gives
$\int_Z\omega=v_B>0$, so $Z$ is nonzero. Continuity of supports and
of $f$ implies $|Z|\subset X_t$. Thus $(t,Z)$ is a point of the
relative cycle space. Irreducible components of a complex space are
closed, so $(t,Z)\in B$. Hence $p_B^{-1}(K)$ is compact. Here one
may use sequential compactness since complex spaces are metrizable
under the usual countability convention.

Remmert's proper mapping theorem makes $p_B(B)$ a closed analytic
subset of $\Delta$. It is irreducible and nonempty; a proper analytic
subset of a disc is discrete. Thus the image is a single point or
the whole disc.
\end{proof}

\subsection{Relative cycles and confinement of negative curves}
\label{sec:barlet}

\begin{proposition}
\label{prop:negative-vertical}
Let $A$ be a holomorphic line bundle on $X$, and assume that
$A|_{X_{t_*}}$ is nef for some $t_*\in\Delta$. If an irreducible
component $B\subset\Cyc_1(X/\Delta)$ contains a cycle of negative
$A$-degree, then $p_B(B)$ is a single point different from $t_*$.
\end{proposition}

\begin{proof}
By Lemma~\ref{lem:periods}, the integer
$d_B=\deg_A Z$ is constant for $(t,Z)\in B$. By hypothesis $d_B<0$.
If $p_B(B)=\Delta$,
there is a cycle $Z_*=\sum_i m_iC_i$ in $X_{t_*}$ represented by a
point of $B$. Analytic nefness implies nonnegative degree on each
irreducible curve, so
\[
 d_B=A|_{X_{t_*}}\cdot Z_*
 =\sum_i m_i\bigl(A|_{X_{t_*}}\cdot C_i\bigr)\ge0,
\]
a contradiction. Proposition~\ref{prop:barlet-proper} gives the stated
single-point alternative, and the same argument excludes $t_*$.
\end{proof}

\begin{corollary}
\label{cor:vertical-germ}
Under the hypotheses of Proposition~\ref{prop:negative-vertical}, let
$Z\subset X_s$ be a nonzero effective cycle with $A\cdot Z<0$.
The reduced analytic germ of $\Cyc_1(X/\Delta)$ at $(s,Z)$ maps
constantly to $s$.
\end{corollary}

\begin{proof}
Each local irreducible branch of this germ is contained in a global
irreducible component through $(s,Z)$. Every such component has
negative degree and, by Proposition~\ref{prop:negative-vertical}, maps
to the single point $s$. There are finitely many local branches.
After choosing a sufficiently small representative of the germ, all
of them map to $s$. On a reduced analytic space, a holomorphic function
which has the constant value $s$ at every point is identically $s$.
\end{proof}

\begin{lemma}\label{lem:barlet-confinement}
Assume $A|_{X_{t_*}}$ is nef. Let $u\colon\Pp^1\to X_s$ be a
holomorphic map with $\deg u^*A<0$. Then the reduced local analytic
space of maps $\Hom(\Pp^1,X)_{\red}$ at $[u]$ parametrizes only maps
whose images lie in $X_s$. On each reduced representative the
universal map factors holomorphically through $X_s$.
\end{lemma}

\begin{proof}
Every holomorphic map $v\colon\Pp^1\to X$ lies in a single fiber of
$f$, because $f\circ v$ is a holomorphic function on $\Pp^1$. Fixing
$z_0\in\Pp^1$, its parameter is the holomorphic function
\[
 \sigma(v)=f(v(z_0))
\]
on the map space. The maps remain nonconstant near $u$:
the integer $\deg v^*A$ is locally constant under continuous
deformation and is negative at $u$.

On a reduced local parameter space $H$ for these maps, the graph of
the universal map is a flat family of compact one-dimensional
subspaces of $\Pp^1\times X$. Passing to its fundamental cycles and
then pushing forward by the proper projection
$\Pp^1\times X\to X$ gives a holomorphic cycle map
\begin{equation}\label{eq:map-cycle}
 H\longrightarrow\Cyc_1(X/\Delta),\qquad
 v\longmapsto\bigl(\sigma(v),v_*[\Pp^1]\bigr).
\end{equation}
The cycle has the generic covering degree as multiplicity on its
rational image. Its $A$-degree is exactly $\deg v^*A<0$; thus the
multiplicities are essential. The graph fibers are smooth and have no embedded components, so the
holomorphic fundamental-cycle map applies
\cite[Theorem~10.2.1]{Mag}. Proper pushforward of these cycles is
holomorphic \cite{Bar75,Bar99}. This uses only the map for graph
families; no properness or surjectivity of a general Douady-to-Barlet
map is being asserted.

Apply Corollary~\ref{cor:vertical-germ} at
$(s,u_*[\Pp^1])$ to \eqref{eq:map-cycle}. It gives $\sigma=s$ on the
reduced germ. Consequently the holomorphic function $f\circ v-s$ on
$H\times\Pp^1$ is zero at every point and hence identically zero.
Here the product is reduced because $H$ is reduced and $\Pp^1$ is
smooth. The universal map therefore factors through the reduced smooth
fiber $X_s$ as a holomorphic map.
\end{proof}

\subsection{External contraction and cone theorems}
We also need the following two theorems on contraction maps from birational geometry on analytic varieties.
\begin{theorem}[Hacon--Xie, smooth case]\label{ext:HX}
Let $M$ be a compact K\"ahler manifold and $\eta$ a K\"ahler class.
If $c_1(K_M)+\eta$ is nef, there are a projective contraction
$g\colon M\to Y$ onto a normal compact K\"ahler space and a
K\"ahler class $\beta$ on $Y$ such that
\begin{equation}\label{eq:HX}
 c_1(K_M)+\eta=g^*\beta\quad\text{in }H^{1,1}_{\BC}(M,\R).
\end{equation}
Here a contraction has $g_*\OO_M=\OO_Y$.
\end{theorem}

This is the specialization of \cite[Theorem~1.4]{HX26} with zero
boundary and nef part $\eta$ descending to $M$. The pair is generalized
klt, and $\eta$ is modified big. Moreover, a smooth manifold is strongly
$\Q$-factorial: every rank-one reflexive sheaf is a line bundle.
Thus the projectivity clause of that theorem applies. In
\cite[Section~2]{HX26}, the symbol $\equiv$ denotes equality in
Bott--Chern cohomology, not merely equality of degrees on curves.
Consequently \eqref{eq:HX} is a cohomological pullback identity, as
required for the hard Lefschetz argument below.

\begin{theorem}[Fujino, compact smooth case]
\label{ext:cone}
Let $h\colon M\to Y$ be a projective morphism between compact complex
spaces, with $M$ smooth. If $K_M$ is not relatively nef, there is an
$h$-contracted rational curve $C$ with
\[
 0<-K_M\cdot C\le 2\dim M.
\]
\end{theorem}

This follows from \cite[Theorem~1.2(5)]{Fuj23} with zero boundary.
The required finite-dimensionality of the relative numerical space
holds for a compact target by \cite[Remark~1.6(iii)]{Fuj23}; no Stein
hypothesis on the whole target is needed.

\section{A fixed boundary class and hard Lefschetz}\label{sec:threshold}

 If a non-nef fiber
exists, an automorphism of the disc sends its parameter to $0$; we
make this change of base coordinate. We retain the known nef fiber
$X_{t_*}$, now with $t_*\ne0$. Fix a K\"ahler form $\omega$ on $X$,
and write $\omega_t=\omega|_{X_t}$ and $M=X_0$. We do not shrink the
disc in a way that discards $t_*$, and we make no assertion about
nefness on the other fibers.

\begin{lemma}[The positive nef threshold]\label{lem:threshold}
Suppose $K_M$ is not nef. Define
\begin{equation}\label{eq:tau}
 \tau=\inf\{s\ge0:c_1(K_M)+s[\omega_0]\in\KK(M)\}.
\end{equation}
Then $0<\tau<\infty$, and
\[
 \alpha_0=c_1(K_M)+\tau[\omega_0]
\]
is nef but not K\"ahler. For $s>\tau$, the class
$c_1(K_M)+s[\omega_0]$ is K\"ahler.
\end{lemma}

\begin{proof}
Choose a smooth real closed representative $\theta$ of $c_1(K_M)$.
Compactness gives $A>0$ with $\theta\ge-A\omega_0$. Thus the set in
\eqref{eq:tau} is nonempty. It is upward closed: adding a positive
multiple of $[\omega_0]$ preserves the K\"ahler property.

If $\tau=0$, a sequence of classes in this set converges to
$c_1(K_M)$, contradicting non-nefness. Hence $\tau>0$. A sequence
$s_j\downarrow\tau$ of admissible parameters shows that $\alpha_0$
is nef. If $\alpha_0$ were K\"ahler, openness of $\KK(M)$ would make
$c_1(K_M)+(\tau-\varepsilon)[\omega_0]$ K\"ahler for some
$0<\varepsilon<\tau$, which is impossible. Finally,
$\alpha_0+(s-\tau)[\omega_0]$ is K\"ahler for $s>\tau$.
\end{proof}

For every $t\in\Delta$, put
\begin{equation}\label{eq:alphat}
 \alpha_t=c_1(K_{X_t})+\tau[\omega_t].
\end{equation}
Then
\begin{equation}\label{eq:one-kahler-class}
 \alpha_{t_*}\in\KK(X_{t_*}),
\end{equation}
because the sum of a nef class and the K\"ahler class
$\tau[\omega_{t_*}]$ is K\"ahler.

\begin{lemma}\label{lem:flatclass}
Under an Ehresmann trivialization of $f$, the classes $\alpha_t$ are
constant in real cohomology, and the cup-product rings are identified.
\end{lemma}

\begin{proof}
The relative canonical bundle
\[
 K_{X/\Delta}=K_X\otimes f^*K_\Delta^{-1}
\]
satisfies $K_{X/\Delta}|_{X_t}\simeq K_{X_t}$. Therefore
\[
 a=c_1(K_{X/\Delta})+\tau[\omega]\in H^2(X,\R)
\]
restricts to $\alpha_t$ on $X_t$.

Ehresmann's theorem makes $f$ a differentiable fiber bundle. Choose a
smooth path from $0$ to any required parameter, in particular to
$t_*$. The pullback bundle over that compact path is smoothly trivial.
If $j_r\colon M\to X$, $0\le r\le1$, are its fiber inclusions, the
maps $j_r$ are homotopic. Thus $j_r^*a=j_0^*a$. Each $j_r$ identifies $M$ diffeomorphically with the corresponding
fiber and preserves the complex orientation. These identifications
therefore preserve the real cohomology rings and their integration
pairings. Equivalently, the restrictions of $a$ form a
flat section of $R^2f_*\R$; on the simply connected disc it has no
monodromy. No global shrinking is needed. Nothing in this argument
identifies the Hodge decompositions or the K\"ahler cones.
\end{proof}

\begin{proposition}
\label{prop:HL}
For every integer $0\le k\le n$, the map
\begin{equation}\label{eq:HL}
 \alpha_0^k\smile(\cdot)\colon H^{n-k}(M,\R)\longrightarrow H^{n+k}(M,\R)
\end{equation}
is an isomorphism. Moreover, $\int_M\alpha_0^n>0$.
\end{proposition}

\begin{proof}
Apply hard Lefschetz \cite{GH78} to the K\"ahler class $\alpha_{t_*}$ on the
single known nef fiber. Its cup-product maps are isomorphisms.
Lemma~\ref{lem:flatclass} identifies those maps with
\eqref{eq:HL}. The same identification gives
\[
 \int_M\alpha_0^n=\int_{X_{t_*}}\alpha_{t_*}^n>0.
\]
This is an identity for a fixed cohomology class, not a limiting
argument for an arbitrary sequence of K\"ahler classes.
\end{proof}

The following argument is the part of the nef-value method that
detects excessively large contracted subvarieties. Compare
\cite[Proposition~3.1]{Wi98}. We give the proof for a possibly singular
analytic target.

\begin{lemma}
\label{lem:vanishing}
Let $M$ be a compact complex manifold of dimension $n$, let
$g\colon M\to Y$ be a holomorphic map to a K\"ahler space, and put
$\alpha=g^*\beta$, where $\beta$ is a K\"ahler class on $Y$.
If $Z\subset M$ is an irreducible compact analytic subvariety of
dimension $d$ and $e=\dim g(Z)$, then
\begin{equation}\label{eq:cupvanish}
 \alpha^{e+1}\smile\PD[Z]=0
 \quad\text{in }H^{2n-2d+2e+2}(M,\R).
\end{equation}
\end{lemma}

\begin{proof}
Choose a K\"ahler form $b$ representing $\beta$ on $Y$ in the
local-potential sense. Its pullback $g^*b$ is a smooth closed form on
$M$. Let $\mu\colon\widetilde Z\to Z$ be a resolution and let
$i\colon Z\hookrightarrow M$ be the inclusion. Set
$h=g\circ i\circ\mu$.

On a dense open subset of $\widetilde Z$, the map $h$ has image in the
smooth locus of $g(Z)$ and has differential of complex rank at most
$e$. Hence $(h^*b)^{e+1}=0$ there. Since $h^*b$ is a smooth form on all
of $\widetilde Z$, this identity holds everywhere. Pushing forward
currents and using the projection formula gives
\[
 (g^*b)^{e+1}\wedge[Z]
 =(i\circ\mu)_*\big((h^*b)^{e+1}\big)=0.
\]
Here the expression on the right means pushforward of the indicated
form against the integration current of $\widetilde Z$; the resolution
has generic degree one, so $(i\circ\mu)_*[\widetilde Z]=[Z]$.
Passing to de Rham cohomology proves \eqref{eq:cupvanish}.
\end{proof}

\begin{proposition}
\label{prop:semismall}
Let $M$ be a connected compact K\"ahler manifold of dimension $n$,
and suppose $\alpha=g^*\beta$ satisfies all the hard Lefschetz
isomorphisms \eqref{eq:HL}. Then every irreducible compact analytic
subvariety $Z\subset M$ satisfies
\begin{equation}\label{eq:semismall}
 2\dim Z-\dim g(Z)\le n.
\end{equation}
\end{proposition}

\begin{proof}
Write $d=\dim Z$ and $e=\dim g(Z)$. Suppose $2d-e>n$, and set
\[
 k=2d-n.
\]
Then $1\le k\le n$ and $k>e$. The class
\[
 \zeta=\PD[Z]\in H^{2n-2d}(M,\R)=H^{n-k}(M,\R)
\]
is nonzero. Indeed, for any K\"ahler form $\omega_M$ on $M$,
\[
 \int_M\zeta\smile[\omega_M]^d=\int_Z\omega_M^d>0.
\]
Lemma~\ref{lem:vanishing} implies
$\alpha^k\smile\zeta=0$, since $k\ge e+1$. This contradicts the
injectivity of the hard Lefschetz map in degree $n-k$.
\end{proof}

\begin{remark}\label{rem:birational}
Taking $Z=M$ in \eqref{eq:semismall} gives $\dim g(M)=n$.
Thus a contraction satisfying its hypotheses is generically finite;
if it has connected fibers and normal target, it is bimeromorphic.
In particular, the contraction arising in our proof cannot be of fiber
type. This observation is not an additional assumption on the
contraction.
\end{remark}

\section{A projective contraction and minimal rational curves}
\label{sec:contraction}

Continue under the assumption that $K_M$ is not nef. Applying
Theorem~\ref{ext:HX} to $\eta=\tau[\omega_0]$ gives
\begin{equation}\label{eq:contraction}
 g\colon M\longrightarrow Y,
 \qquad \alpha_0=g^*\beta,
\end{equation}
where $g$ is projective and $Y$ is normal and compact K\"ahler.
The contraction is not an isomorphism, since otherwise $\alpha_0$
would be K\"ahler. It therefore has a positive-dimensional fiber:
a finite contraction with $g_*\OO_M=\OO_Y$ is an isomorphism.
Propositions~\ref{prop:HL} and \ref{prop:semismall} apply to $g$.

\begin{lemma}
\label{lem:nullcurve}
For an irreducible curve $D\subset M$,
\[
 \alpha_0\cdot D=0\quad\Longleftrightarrow\quad g(D)\text{ is a point}.
\]
Every contracted curve satisfies
\begin{equation}\label{eq:degree-area}
 -K_M\cdot D=\tau\int_D\omega_0>0.
\end{equation}
Moreover, $L=-K_M$ is $g$-ample.
\end{lemma}

\begin{proof}
If $g(D)$ is a point, the pullback class $g^*\beta$ has degree zero on
$D$. If $g(D)$ is a curve, the map from the normalization of $D$ to
$g(D)$ has positive generic degree, so the integral of the pullback
of a K\"ahler form representing $\beta$ is strictly positive.
This proves the equivalence. Equation~\eqref{eq:degree-area} follows
from \eqref{eq:contraction} and the definition of $\alpha_0$.

For completeness, relative ampleness is stronger than strict
positivity on individual curves, so it deserves a separate argument.
Every fiber $F=g^{-1}(y)$ is projective. For every positive-dimensional
irreducible subvariety $W\subset F_{\red}$, with $q=\dim W$, the
restriction of $g^*\beta$ to $W$ is zero. Therefore
\begin{equation}\label{eq:nakai}
 c_1(L)^q\cdot[W]=\tau^q\int_W\omega_0^q>0.
\end{equation}
The Nakai--Moishezon criterion on the projective complex space $F_{\red}$ shows
that $L|_{F_{\red}}$ is ample. Ampleness is unaffected by nilpotents,
so $L|_F$ is ample as well. Fiberwise ampleness implies relative
ampleness in a neighborhood of each $y$; the precise analytic
statement is \cite[Lemma~3.3]{Fuj23}. Hence $L$ is $g$-ample.
\end{proof}

\begin{lemma}
\label{lem:minimal}
There is a $g$-contracted rational curve. Consequently the integer
\begin{equation}\label{eq:ell}
 \ell=\min\{-K_M\cdot C:
 C\subset M\text{ is rational and }g(C)\text{ is a point}\}
\end{equation}
is defined and positive.
\end{lemma}

\begin{proof}
A positive-dimensional projective fiber contains an irreducible
curve $D$, obtained, for example, by successive hyperplane sections
of one of its positive-dimensional components.
Equation~\eqref{eq:degree-area} gives $K_M\cdot D<0$.
Thus $K_M$ is not relatively nef for $g$. Theorem~\ref{ext:cone}
gives a contracted rational curve with positive anticanonical degree.
All such degrees are positive integers because $K_M$ is a line
bundle. A nonempty set of positive integers has a minimum.
\end{proof}

Fix a curve $C$ attaining \eqref{eq:ell}, and let
\begin{equation}\label{eq:nu}
 \nu\colon\Pp^1\longrightarrow C\subset M
\end{equation}
be its normalization.

Let $\mathcal H\subset\Hom(\Pp^1,M)$ be the locus, with its reduced
structure, of holomorphic maps $u\colon\Pp^1\to M$ with
$\deg u^*L=\ell$ and $g\circ u$ constant.

\begin{lemma}
\label{lem:unsplit}
Every map $u\colon\Pp^1\to M$ with $\deg u^*L=\ell$ and $g\circ u$
constant is birational onto its image and has trivial automorphism
group.
\end{lemma}

\begin{proof}
Let $u\colon\Pp^1\to M$ be as in the statement. Since $\ell>0$, the map
$u$ is nonconstant; let $C'$ be its reduced image and $m$ its generic
mapping degree. Then $C'$ is an irreducible rational curve contracted
by $g$, so minimality of $\ell$ gives $L\cdot C'\ge\ell$. Hence
\begin{equation}\label{eq:nosplitting}
 \ell=\deg u^*L=m(L\cdot C')\ge m\ell,
\end{equation}
so $m=1$: the map $u$ is birational onto its image.

A birational map from $\Pp^1$ is its image's normalization. Any
automorphism $\sigma\in\Aut(\Pp^1)$ satisfying $u\circ\sigma=u$ fixes
a dense open subset of $\Pp^1$, and hence is the identity.
\end{proof}

Let
\begin{equation}\label{eq:relative-moduli}
 \Mspace=\mathcal H/\PGL_2(\C)
\end{equation}
denote the orbit space for reparametrization.  By
Proposition~\ref{prop:minimal-moduli} in Appendix~\ref{app:moduli},
the action is free and proper, $\Mspace$ is a compact complex analytic
space with a universal smooth family of rational curves and a
holomorphic evaluation map, and every fiber of $\Mspace\to Y$ is
projective.  It is nonempty by Lemma~\ref{lem:minimal}.

Let $V\subset\Mspace$ be any irreducible component, with its reduced
structure. Pulling back the universal family gives
\begin{equation}\label{eq:universal}
 p\colon U\longrightarrow V,
 \qquad q\colon U\longrightarrow M,
 \qquad a\colon V\longrightarrow Y,
 \qquad g\circ q=a\circ p.
\end{equation}
The map $p$ is smooth and proper with fibers $\Pp^1$.
Since $V$ is irreducible and the fibers are connected and irreducible,
$U$ is irreducible. For $r=\dim V$, we have $\dim U=r+1$.
The restriction $q|_{p^{-1}(v)}$ is the normalization of the curve
$C_v$ represented by $v$.

We write
\[
 Z=\Locus(V):=q(U).
\]
The space $U$ is compact, so Remmert's theorem makes $Z$ an
irreducible compact analytic subvariety of $M$.

For $x\in M$ set
\[
 V_x=\{v\in V:x\in C_v\},
 \qquad
 V_{x,z}=V_x\cap V_z.
\]
These are closed analytic subspaces: for example,
$V_x=p(q^{-1}(x))$ is analytic by properness. Since each curve is
normalized by its fiber of $p$, the map $q^{-1}(x)\to V_x$ is finite
and surjective.

\begin{lemma}\label{lem:two-points}
For distinct points $x,z\in M$, the set $V_{x,z}$ is finite.
\end{lemma}

\begin{proof}
There is nothing to prove if $V_{x,z}$ is empty. Otherwise all its
members are contained in the fixed projective fiber
$F=g^{-1}(g(x))$, and necessarily $g(z)=g(x)$.
Moreover, $V_{x,z}$ is a closed analytic subspace of the projective
moduli space $\Mspace_{g(x)}$. It is therefore projective.

Suppose it had positive dimension. Choose an irreducible projective
curve in $V_{x,z}$ and normalize it; call the resulting smooth
projective curve $B$. Its map to $\Mspace$ is nonconstant. Pulling
back the universal family gives a smooth proper $\Pp^1$-family
\[
 p_B\colon S_B\longrightarrow B
\]
and a morphism $e_B\colon S_B\to F$. The inverse images of $x$ and
$z$ are finite and surjective over $B$. After taking an irreducible
component of their fiber product dominating $B$ and normalizing, we obtain
a finite surjective base change $B'\to B$ for which the pulled-back
family $p'\colon S'\to B'$ has two sections $\sigma_x$ and $\sigma_z$.
They are disjoint because they map to distinct points $x$ and $z$.

The surface $S'$ is a ruled surface. Indeed, the divisor of either
section has degree one on the fibers. Its direct image is a rank-two
vector bundle $E$ on $B'$, because the fiberwise $H^1$ vanishes and
$h^0=2$. The evaluation map defines an isomorphism
$S'\simeq\Pp(E)$, as can be checked on each fiber. The holomorphic
bundle $E$ on the smooth projective curve $B'$ is algebraic by GAGA,
so $S'$ is smooth and projective. Let
$e'\colon S'\to F$ be the evaluation morphism.

Its image $T=e'(S')$ is a surface. If it were a curve, the images of
all fibers of $p'$ would be the same irreducible curve: a projective
curve has only finitely many irreducible components, and
the parameter curve is irreducible. Each image has a unique
normalization map up to isomorphism, so the map $B'\to\Mspace$ would
be constant, a contradiction. Therefore $e'\colon S'\to T$ is
generically finite.

Choose an ample line bundle $H$ on $F$ and put $D=c_1(e'^*H)$.
Then $D$ is nef and
\[
 D^2=\deg(e'\colon S'\to T)\,(H^2\cdot[T])>0.
\]
Both sections are contracted by $e'$, hence
\[
 D\cdot\sigma_x=D\cdot\sigma_z=0.
\]
The Hodge index theorem implies
\begin{equation}\label{eq:negative-sections}
 \sigma_x^2<0,\qquad \sigma_z^2<0,
\end{equation}
since neither section is numerically zero: each meets a ruling fiber
in one point.

On a ruled surface, two sections differ numerically by a multiple of
a ruling fiber $F_p$. To verify the point being used, the line bundle
$\OO_{S'}(\sigma_z-\sigma_x)$ has degree zero on every ruling fiber,
and its restriction there is therefore trivial. Its direct image is
a line bundle $Q$ on $B'$; cohomology and base change and the evaluation
map identify $\OO_{S'}(\sigma_z-\sigma_x)$ with $p'^*Q$. Numerically
$p'^*Q$ is $\deg(Q)$ times a ruling fiber. Thus
\[
 \sigma_z\equiv\sigma_x+bF_p
\]
for some real number $b$, with $F_p^2=0$ and
$\sigma_x\cdot F_p=1$. Disjointness gives
$0=\sigma_x\cdot\sigma_z=\sigma_x^2+b$. Consequently
\[
 \sigma_z^2=\sigma_x^2+2b=-\sigma_x^2>0,
\]
contrary to \eqref{eq:negative-sections}. This proves that $V_{x,z}$
is zero-dimensional. Its compactness then makes it finite.
\end{proof}

This is the required form of bend-and-break, proved entirely inside
a projective contraction fiber. Compare \cite{Kol96} and
\cite[Section~3]{Wi98}. Neither $M$ nor $V$ is assumed projective.

\begin{proposition}\label{prop:incidence}
Let $V$, $U$, $Z$ be as above, and set
\[
 r=\dim V,\qquad d=\dim Z,\qquad e=\dim g(Z).
\]
Then
\begin{equation}\label{eq:incidence}
 r+2\le 2d-e.
\end{equation}
\end{proposition}

\begin{proof}
Choose a general point $x\in Z$ so that both of the following hold:
\[
 \dim q^{-1}(x)=r+1-d,
 \qquad
 \dim\bigl(Z\cap g^{-1}(g(x))\bigr)=d-e.
\]
Such points exist by the fiber-dimension theorem, applied to the two
proper surjective maps $q\colon U\to Z$ and $g|_Z\colon Z\to g(Z)$.
More explicitly, intersect the open set of minimal fiber dimension
for $q$ with the inverse image of the corresponding open set in
$g(Z)$.

Take an irreducible component $A$ of $q^{-1}(x)$ with dimension
$r+1-d$, and set $B=p(A)\subset V_x$. Since $A\to B$ is finite,
\[
 \dim B=r+1-d.
\]
The family $p^{-1}(B)\to B$ is again a smooth $\Pp^1$-family, with
irreducible total space of dimension $r+2-d$. Denote its swept locus
by $W=q(p^{-1}(B))$.

For every $z\in W\setminus\{x\}$, there are only finitely many
curves from $B$ through $z$, by Lemma~\ref{lem:two-points}. Each such
curve has only finitely many preimages of $z$ on its normalization.
Thus the evaluation morphism $p^{-1}(B)\to W$ has finite fibers away
from $x$. Since its image contains a nonconstant curve, it is
generically finite. Hence
\[
 \dim W=\dim p^{-1}(B)=r+2-d.
\]
Every curve in $B$ passes through $x$ and is contracted by $g$, so
\[
 W\subset Z\cap g^{-1}(g(x)).
\]
It follows that $r+2-d\le d-e$, which is \eqref{eq:incidence}.
\end{proof}

\section{Deformation theory in the total space}\label{sec:deformation}
The following estimates of dimensions had been stated in \cite[1.17 Remark]{Kol96}. We include here a proof for convenience.
\begin{lemma}
\label{lem:kuranishi}
Let $Q$ be a complex manifold and $u\colon\Pp^1\to Q$ a holomorphic
map. The local analytic space of holomorphic maps satisfies
\begin{equation}\label{eq:deformation-estimate}
 \dim_{[u]}\Hom(\Pp^1,Q)
 \ge h^0(\Pp^1,u^*T_Q)-h^1(\Pp^1,u^*T_Q)
 =\dim Q+\deg u^*T_Q.
\end{equation}
This statement does not require $Q$ to be compact or algebraic.
\end{lemma}

\begin{proof}
Put $E=u^*T_Q$. Since the source is compact, the analytic space
$H=\Hom(\Pp^1,Q)$ exists as an open graph locus in a Douady space;
see \cite[Notation~2.2]{KKL10}. Its tangent space at $[u]$ is $H^0(E)$.
The vector space $H^1(E)$ is a complete obstruction space for fixed-source,
fixed-target deformations. Indeed, for a small extension of local
Artinian $\C$-algebras $B'\to B$ with kernel $J$ annihilated by the
maximal ideal of $B'$, a deformation over $B$ lifts locally on the
source because $Q$ is smooth. Differences of local lifts form a
\v{C}ech cocycle with values in $E\otimes J$. Its class in
$H^1(E)\otimes J$ vanishes precisely when the local lifts can be
modified to glue. This construction is compatible with quotients of
$J$; compare the local-lifting argument in \cite[Section~2.E]{KKL10}.

For completeness, the relation between this obstruction space and
actual local dimension can be seen formally, without assuming the
convergence of a chosen Kuranishi obstruction map. Let
$R=\OO_{H,[u]}$, put $a=h^0(E)$, and take a minimal presentation
\[
 \widehat R=S/I,\qquad
 S=\C[[x_1,\ldots,x_a]],\qquad I\subset\mathfrak m^2,
\]
where $\mathfrak m=(x_1,\ldots,x_a)$. We claim that the number of
minimal generators of $I$ is at most $h^1(E)$. By Artin--Rees, choose
$N$ sufficiently large that $I\cap\mathfrak m^N\subset\mathfrak m I$.
The canonical deformation over
$B=S/(I+\mathfrak m^N)$ has an obstruction to lifting to
\[
 B'=S/(\mathfrak m I+\mathfrak m^N),\qquad
 J=\ker(B'\to B)\simeq I/\mathfrak m I.
\]
This obstruction determines a linear map $J^*\to H^1(E)$.
It is injective. Otherwise, a nonzero $\lambda\in J^*$ in its
kernel would, by completeness of the obstruction space, give a lift
of $\widehat R\to B$ to the pushout extension
$B'_{\lambda}=B'/\ker\lambda$. Its one-dimensional kernel is generated
by an element $\varepsilon$ annihilated by the maximal ideal.
Any such lift sends each $x_i$ to $x_i+c_i\varepsilon$.
Since $I\subset\mathfrak m^2$, this substitution does not change the
image of any $f\in I$ in $B'_{\lambda}$. A relation with
$\lambda([f])\ne0$ therefore remains nonzero, a contradiction.
Thus $\dim I/\mathfrak m I\le h^1(E)$.

Nakayama's lemma and the height bound for an ideal generated by this
many elements give
\[
 \dim_{[u]}H=\dim R=\dim\widehat R
 \ge a-h^1(E)=h^0(E)-h^1(E).
\]
Riemann--Roch for the vector bundle $E$ on $\Pp^1$ gives
$\chi(E)=\rk E+\deg E$, proving the stated equality. This estimates
the dimension of the analytic germ itself, not merely its tangent
space; passing to the reduced germ leaves its dimension unchanged.
All constructions take place near the compact graph of $u$, and no
compactness of $Q$ or vanishing of $H^1(E)$ is required.
\end{proof}

 View the normalization \eqref{eq:nu} as a map into $X$. Since $M$ is
a smooth fiber, its normal bundle in $X$ is trivial and adjunction gives
\begin{equation}\label{eq:adjunction}
 K_X|_M\simeq K_M.
\end{equation}
One may also obtain this from $K_\Delta\simeq\OO_\Delta$ and the
relative canonical bundle. Consequently
\[
 \deg\nu^*T_X=-K_X\cdot C=-K_M\cdot C=\ell.
\]
Lemma~\ref{lem:kuranishi} gives
\begin{equation}\label{eq:total-lower}
 \dim_{[\nu]}\Hom(\Pp^1,X)\ge n+1+\ell.
\end{equation}

\begin{lemma}\label{lem:confinement}
Assume $K_{X_{t_*}}$ is nef. Every sufficiently small deformation of
$\nu$ as a holomorphic map into $X$ has image in $M$, is contracted
by $g$, and has degree $\ell$ with respect to $L=-K_M$. It is
birational onto its image. These assertions hold on the entire reduced
local map germ, not just on a selected smooth family of maps.
\end{lemma}

\begin{proof}
Apply Lemma~\ref{lem:barlet-confinement} to the line bundle
$A=K_{X/\Delta}$ and to $\nu$. Its degree is
\[
 \deg\nu^*A=K_M\cdot C=-\ell<0,
\]
and $A|_{X_{t_*}}=K_{X_{t_*}}$ is nef. The entire reduced local germ
therefore factors through $M$. This is exactly where the additional
Barlet argument replaces nefness on all noncentral fibers.

The nearby maps, now regarded as maps to $M$, are homotopic to $\nu$
in $M$. For example, after choosing a Riemannian metric on the compact
manifold $M$, sufficiently uniformly close maps are joined pointwise
by the short geodesics in a fixed injectivity-radius neighborhood.
Thus for every sufficiently close map $u$,
\[
 \int_{\Pp^1}u^*\alpha_0=0,
 \qquad \deg u^*L=\ell.
\]
Since $\alpha_0=g^*\beta$ and $\beta$ is K\"ahler, a nonconstant
map $g\circ u$ would have strictly positive $\beta$-area. Hence
$g\circ u$ is constant. The map $u$ is nonconstant because its
$L$-degree is positive. Its image $C'$ is rational: a nonconstant
map from $\Pp^1$ factors through the normalization of $C'$, whose
genus is zero by Riemann--Hurwitz. If its generic degree is $m$, then
\[
 \ell=\deg u^*L=m(L\cdot C')\ge m\ell.
\]
Minimality of $\ell$ therefore gives $m=1$. Every nearby map defines
a point of the moduli space $\Mspace$.
\end{proof}

\begin{proposition}
\label{prop:large-family}
There is an irreducible component $V\subset\Mspace$ for which
\begin{equation}\label{eq:large-family}
 \dim V\ge n+\ell-2.
\end{equation}
\end{proposition}

\begin{proof}
Choose an irreducible component of the reduced local germ
of $\Hom(\Pp^1,X)$ at $[\nu]$, and let $H$ be a sufficiently
small representative having dimension at least
$n+1+\ell$, as supplied by \eqref{eq:total-lower}. Shrink it around
$[\nu]$. Lemma~\ref{lem:confinement} gives a holomorphic map
\[
 H\longrightarrow\Mspace
\]
by forgetting the parametrization. To be precise about the analytic
factorization, on a reduced parameter germ the function
$f\circ u$ vanishes identically because it vanishes at all its
points. Thus the universal map factors through $M$; its composite
with $g$ is constant on each source fiber and gives the required
holomorphic parameter in $Y$.

Two parametrized maps in $H$ represent the same point of $\Mspace$
exactly when they differ by an automorphism of $\Pp^1$. Every map
here is birational onto its image and has trivial stabilizer. The
fibers in the full parametrized-map space are $\PGL_2(\C)$-orbits
of dimension $3$. Their intersections with $H$ therefore have
dimension at most $3$, which is the upper bound needed here.

After shrinking the irreducible germ $H$, its image is contained in
one irreducible component $V$ of $\Mspace$: locally the target has
only finitely many irreducible components, and their inverse images
are closed analytic subsets covering the irreducible germ. The
fiber-dimension inequality for the holomorphic map $H\to V$, whose
fibers have dimension at most $3$, gives
\[
 \dim V\ge\dim H-3\ge n+\ell-2.
\]
This reasoning does not require the image of the nonproper map
$H\to V$ to be a closed analytic subset. The chosen component $V$ is
itself compact because $\Mspace$ is compact.
\end{proof}

\section{Proof of the main theorem}\label{sec:proof}

\begin{proof}[Proof of Theorem~\ref{thm:main}]
Assume a fiber with non-nef canonical bundle exists. After an
automorphism of $\Delta$, write it as $M=X_0$, and keep the parameter
$t_*\ne0$ of the known nef fiber. Construct $\tau>0$ and
$\alpha_0=c_1(K_M)+\tau[\omega_0]$ as in
Lemma~\ref{lem:threshold}. The class $\alpha_0$ is nef but not
K\"ahler. Proposition~\ref{prop:HL}, using only the fiber $X_{t_*}$,
shows that it satisfies all hard Lefschetz isomorphisms.

By Theorem~\ref{ext:HX}, there is a projective contraction
$g\colon M\to Y$ with $\alpha_0=g^*\beta$ for a K\"ahler class
$\beta$ on $Y$. It is nontrivial. Proposition~\ref{prop:semismall}
gives, for every irreducible compact analytic $Z\subset M$,
\[
 2\dim Z-\dim g(Z)\le n.
\]

Lemmas~\ref{lem:nullcurve} and \ref{lem:minimal} provide the
$g$-ample line bundle $L=-K_M$ and a minimal contracted rational
curve of degree $\ell\in\Z_{>0}$. The moduli space $\Mspace$ of
degree-$\ell$ maps contracted by $g$ is compact and consists only of
birational maps from smooth $\Pp^1$'s, by Lemma~\ref{lem:unsplit} and
Appendix~\ref{app:moduli}.

Propositions~\ref{prop:barlet-proper} and
\ref{prop:negative-vertical}, followed by
Lemma~\ref{lem:barlet-confinement}, confine the reduced total-space
map germ of a minimal contracted curve to $M$. Accordingly,
Proposition~\ref{prop:large-family} gives an irreducible component
$V\subset\Mspace$ with $r=\dim V\ge n+\ell-2$. Put
$Z=\Locus(V)$, $d=\dim Z$, and $e=\dim g(Z)$. The incidence
inequality of Proposition~\ref{prop:incidence} now yields
\begin{equation}\label{eq:final}
 n+\ell\ \le\ r+2\ \le\ 2d-e\ \le\ n.
\end{equation}
Since $\ell>0$, this is impossible. The assumed non-nef fiber cannot
exist. Hence $K_{X_t}$ is nef for every $t\in\Delta$.
\end{proof}

\appendix
\section{The compact moduli space of minimal rational curves}
\label{app:moduli}

We justify here the quotient used in Section~\ref{sec:contraction}.
The point that requires care is that
$\PGL_2(\C)$ is noncompact and the map from parametrized maps to their
image cycles is not proper.  Compactness is obtained instead from
Gromov compactness together with the minimality of the anticanonical
degree.  Minimality eliminates every possible bubble in a stable
limit.

Keep the notation of Section~\ref{sec:contraction}.  Thus
$M=X_0$ is compact K\"ahler,
\[
 g\colon M\longrightarrow Y,\qquad
 \alpha_0=c_1(K_M)+\tau[\omega_0]=g^*\beta,
\]
$L=-K_M$ is $g$-ample, and
\[
 \ell=\min\{L\cdot C:C\subset M\text{ rational and }g(C)
 \text{ is a point}\}>0.
\]
Let
\[
 \mathcal H=\{u\in\Hom(\Pp^1,M):\deg u^*L=\ell,
 \ g\circ u\text{ is constant}\}_{\red}.
\]
We let $G=\PGL_2(\C)$ act on $\mathcal H$ by
\begin{equation}\label{eq:G-action}
 \sigma\cdot u=u\circ\sigma^{-1}.
\end{equation}

\begin{proposition}\label{prop:minimal-moduli}
With the notation above, the following statements hold.
\begin{enumerate}
 \item[(a)] $\mathcal H$ is a reduced complex analytic space and the
 action \eqref{eq:G-action} is holomorphic, free, and proper.
 Consequently the orbit space
 \[
  \Mspace:=\mathcal H/G
 \]
 carries a unique complex-space structure for which
 $\pi\colon\mathcal H\to\Mspace$ is a holomorphic principal
 $G$-bundle.
 \item[(b)] The complex space $\Mspace$ is compact.  There are a
 holomorphic map $a\colon\Mspace\to Y$, a holomorphically locally
 trivial $\Pp^1$-bundle
 \[
  p\colon U\longrightarrow\Mspace,
 \]
 and a holomorphic evaluation map $q\colon U\to M$ satisfying
 $g\circ q=a\circ p$.  For every $v\in\Mspace$, the map
 $q|_{p^{-1}(v)}$ is the normalization of the rational curve
 represented by $v$.
 \item[(c)] Every fiber $\Mspace_y=a^{-1}(y)$ is a projective complex
 space.
\end{enumerate}
\end{proposition}

\begin{proof}
\emph{Step 1: analytic structure and freeness.}
The degree $\deg u^*L$ is locally constant on
$\Hom(\Pp^1,M)$.  On the positive-degree locus, the fundamental-cycle
map
\[
 \Phi\colon\Hom(\Pp^1,M)\longrightarrow\Cyc_1(M),\qquad
 u\longmapsto u_*[\Pp^1],
\]
is holomorphic; this is the same graph-to-cycle construction used in
\eqref{eq:map-cycle}.  The condition that $g\circ u$ be constant is
therefore the inverse image of the relative cycle space
$\Cyc_1(M/Y)$.  Hence $\mathcal H$ is a reduced complex analytic
space.

For $u\in\mathcal H$, let $C_u$ be its reduced image and let $m$ be
the generic mapping degree.  Since $C_u$ is a contracted rational
curve, minimality of $\ell$ gives
\[
 \ell=\deg u^*L=m(L\cdot C_u)\ge m\ell.
\]
Thus $m=1$.  Hence every $u\in\mathcal H$ is the normalization of its
image and has trivial stabilizer in $G$.  The action is free.

\emph{Step 2: properness of the reparametrization action.}
Consider
\[
 \Theta\colon G\times\mathcal H\longrightarrow
 \mathcal H\times\mathcal H,
 \qquad (\sigma,u)\longmapsto(u,\sigma\cdot u).
\]
We prove that $\Theta$ is proper.  Since complex spaces are metrizable
under the usual countability convention, it is enough to use the
sequential criterion.  Suppose
\[
 u_j\longrightarrow u,
 \qquad
 v_j:=\sigma_j\cdot u_j
       =u_j\circ\sigma_j^{-1}\longrightarrow v
\]
in $\mathcal H$.  For every $j$, $u_j$ and $v_j$ have the same image
cycle.  Continuity of the cycle map gives
\[
 u_*[\Pp^1]=v_*[\Pp^1].
\]
Both $u$ and $v$ are birational normalizations of this same reduced
curve.  Hence there is a unique $\sigma\in G$ such that
$v=\sigma\cdot u$, equivalently $u=v\circ\sigma$.

Choose three distinct points $z_1,z_2,z_3\in\Pp^1$ whose images
$u(z_i)$ are distinct smooth points of the image curve.  Since
\[
 v_j(\sigma_j z_i)=u_j(z_i),
\]
any cluster point $w$ of the sequence $\sigma_j z_i$ satisfies
$v(w)=u(z_i)$.  A normalization has exactly one preimage over a smooth
point, so necessarily $w=\sigma z_i$.  Thus
\[
 \sigma_j z_i\longrightarrow\sigma z_i,
 \qquad i=1,2,3.
\]
The three limiting points are distinct.  A M\"obius transformation is
uniquely determined by the images of three distinct points, and the
corresponding three-point coordinate on $\PGL_2(\C)$ is continuous.
Consequently $\sigma_j\to\sigma$ in $G$.  This proves properness of
$\Theta$.

A free proper holomorphic action of a complex Lie group on a complex
space has a complex analytic geometric quotient; see
\cite[Lemma~3.2.6]{HH99}.
Therefore $\Mspace=\mathcal H/G$ is a complex space and
$\pi\colon\mathcal H\to\Mspace$ is a holomorphic principal
$G$-bundle.

\emph{Step 3: compactness of the quotient.}
For every $u\in\mathcal H$, the curve $u(\Pp^1)$ is contracted by
$g$.  Since
$\alpha_0=c_1(K_M)+\tau[\omega_0]=g^*\beta$ and
$L=-K_M$, we have
\begin{equation}\label{eq:minimal-energy}
 0=\int_{\Pp^1}u^*\alpha_0
  =-\ell+\tau\int_{\Pp^1}u^*\omega_0,
 \qquad
 \int_{\Pp^1}u^*\omega_0=\frac{\ell}{\tau}.
\end{equation}
Thus all maps in $\mathcal H$ have the same $\omega_0$-energy.

Take an arbitrary sequence $[u_j]\in\Mspace$ and choose
representatives $u_j\in\mathcal H$.  Gromov compactness for
holomorphic spheres in the compact K\"ahler manifold
$(M,\omega_0)$ gives, after passing to a subsequence, a genus-zero
stable holomorphic limit of the $u_j$; see
\cite[Chapters~4--5]{MS12}.  The total $L$-degree of the stable limit
is $\ell$.  Moreover, because every $g\circ u_j$ is constant, every
component of the stable limit is mapped by $g$ to one and the same
point of $Y$.

Let $u_\lambda\colon C_\lambda\simeq\Pp^1\to M$ be a nonconstant
component of the stable limit.  Its image is a contracted rational
curve.  If $m_\lambda$ is its generic mapping degree, then
\[
 \deg u_\lambda^*L
 =m_\lambda\bigl(L\cdot u_\lambda(C_\lambda)_{\red}\bigr)
 \ge\ell.
\]
Since the sum of the $L$-degrees of all nonconstant components is
exactly $\ell$, there is precisely one nonconstant component, its
$L$-degree is $\ell$, and it is birational onto its image.

There are no constant components either.  Indeed, if there were $m>0$
constant components, the dual graph would be a tree with $m+1$
vertices and hence $m$ edges.  Stability of an unmarked stable map
requires every constant component to have valence at least three.
The sum of the valences of the constant vertices would therefore be
at least $3m$, whereas the sum of the valences of \emph{all} vertices
is $2m$, a contradiction.  Hence the stable limit has a single smooth
domain $\Pp^1$ and belongs again to $\mathcal H$.

In the absence of bubbles, Gromov convergence is ordinary smooth
convergence after reparametrization.  Thus there are
$\gamma_j\in G$ and $u_\infty\in\mathcal H$ such that
\[
 \gamma_j\cdot u_j\longrightarrow u_\infty
 \quad\text{in }\Hom(\Pp^1,M).
\]
Therefore $[u_j]\to[u_\infty]$ in $\Mspace$.  We have proved that
$\Mspace$ is sequentially compact.  Since it is a second-countable
Hausdorff complex space, it is metrizable; hence it is compact.

\emph{Step 4: the universal family.}
The holomorphic map
\[
 a_{\mathcal H}\colon\mathcal H\longrightarrow Y,
 \qquad a_{\mathcal H}(u)=g(u(z_0))
\]
is independent of the chosen $z_0\in\Pp^1$ and is $G$-invariant.
It descends to a holomorphic map $a\colon\Mspace\to Y$.

Let $G$ act diagonally on $\Pp^1\times\mathcal H$ by
\begin{equation}\label{eq:diagonal-action}
 \sigma\cdot(z,u)=(\sigma z,\sigma\cdot u)
                 =(\sigma z,u\circ\sigma^{-1}).
\end{equation}
Since $\mathcal H\to\Mspace$ is a principal $G$-bundle, the quotient
\[
 U=(\Pp^1\times\mathcal H)/G
\]
is the associated $\Pp^1$-bundle over $\Mspace$ and is therefore
holomorphically locally trivial.  The evaluation map
\[
 \operatorname{ev}\colon\Pp^1\times\mathcal H\to M,
 \qquad (z,u)\longmapsto u(z),
\]
is invariant under \eqref{eq:diagonal-action}, because
\[
 (u\circ\sigma^{-1})(\sigma z)=u(z).
\]
It descends to $q\colon U\to M$.  By construction
$g\circ q=a\circ p$, and on every fiber of $p$ the map $q$ is the
normalization of the corresponding rational curve.

\emph{Step 5: projectivity of the fibers.}
The cycle map $\Phi$ is $G$-invariant and descends to a holomorphic
map
\[
 \overline\Phi\colon\Mspace\longrightarrow\Cyc_1(M/Y).
\]
It is injective: two birational maps from $\Pp^1$ with the same image
cycle are two normalizations of the same curve and hence differ by a
unique element of $G$.  Since $\Mspace$ is compact,
$\overline\Phi$ is proper; being proper with finite fibers, it is a
finite holomorphic map onto its image.

Fix $y\in Y$.  By Remmert's theorem the image
\[
 \mathcal R_y:=\overline\Phi(\Mspace_y)
\]
is a closed analytic subspace of the degree-$\ell$ cycle space of the
projective fiber $M_y$.  Because $L|_{M_y}$ is ample, this degree
piece is the analytification of a projective Chow variety; hence
$\mathcal R_y$ is projective by Chow's theorem.  The finite analytic
map
\[
 \Mspace_y\longrightarrow\mathcal R_y
\]
algebraizes by GAGA: the coherent
$\mathcal O_{\mathcal R_y}$-algebra
$(\overline\Phi_y)_*\mathcal O_{\Mspace_y}$ is algebraic, and its
relative spectrum is finite, hence projective, over
$\mathcal R_y$.  Its analytification is $\Mspace_y$.  Thus
$\Mspace_y$ is projective.
\end{proof}

{\small {School of Mathematics, Hunan University, China

 {\it E-mail address}: mulin@hnu.edu.cn}

 {School of Mathematical Sciences, Dalian University of Technology,  China

 {\it E-mail address}: xlliu1124@dlut.edu.cn}

{School of Mathematics and statistics, Wuhan  University, Wuhan 430072, China

{\it E-mail address}: likeanyone@whu.edu.cn}}


\begin{thebibliography}{99}

\bibitem{AP97}
M.~Andreatta and T.~Peternell,
\emph{On the limits of manifolds with nef canonical bundles},
in Complex Analysis and Geometry (Trento, 1995),
Pitman Research Notes in Mathematics Series \textbf{366},
Longman, Harlow, 1997, 1--6.


\bibitem{Bar75}
D.~Barlet,
\emph{Espace analytique r\'eduit des cycles analytiques complexes
compacts d'un espace analytique complexe de dimension finie},
in Fonctions de plusieurs variables complexes II,
Lecture Notes in Mathematics \textbf{482}, Springer, 1975, 1--158.

\bibitem{Bar99}
D.~Barlet,
\emph{How to use the cycle space in complex geometry},
in Several Complex Variables, MSRI Publications \textbf{37},
Cambridge University Press, 1999, 25--42.


\bibitem{Bishop}
E.~Bishop,
\emph{Conditions for the analyticity of certain sets},
Michigan Mathematical Journal \textbf{11} (1964), no.~4, 289--304.

\bibitem{CP99} F.~Campana, T.~Peternell,
\emph{Recent developments in the classification theory of compact
K\"ahler manifolds},
in Several Complex Variables (Berkeley, CA, 1995--1996), 113--159,
Math.\ Sci.\ Res.\ Inst.\ Publ.\ \textbf{37},
Cambridge University Press, Cambridge, 1999.

\bibitem{F1} Fujiki, A. \emph{Closedness of the Douady spaces of compact Kahler spaces}, Publ. Res. Inst. Math. Sci. 14 (1978/79), no. 1, 1-52.




\bibitem{HH99}
P.~Heinzner and A.~Huckleberry,
\emph{Analytic Hilbert quotients},
in Several Complex Variables,
Math. Sci. Res. Inst. Publ. \textbf{37},
Cambridge University Press, 1999, 309--349.

\bibitem{MS12}
D.~McDuff and D.~Salamon,
\emph{$J$-holomorphic Curves and Symplectic Topology},
2nd ed., American Mathematical Society Colloquium Publications,
vol.~52, American Mathematical Society, Providence, RI, 2012.

\bibitem{Ser56}
J.-P.~Serre,
\emph{G\'eom\'etrie alg\'ebrique et g\'eom\'etrie analytique},
Annales de l'Institut Fourier \textbf{6} (1955--1956), 1--42.

\bibitem{Fuj23}
O.~Fujino,
\emph{Cone and contraction theorem for projective morphisms between
complex analytic spaces},
preprint, version~4, August~11, 2023.
\arxiv{2209.06382v4}.


\bibitem{GH78}
P.~Griffiths and J.~Harris,
\emph{Principles of Algebraic Geometry},
Wiley-Interscience, New York, 1978.

\bibitem{HX26}
C.~Hacon and L.~Xie,
\emph{On the K\"ahler MMP and the transcendental base-point-free
theorem},
preprint, \arxiv{2607.24986v1}, 2026.


\bibitem{KKL10}
S.~Kebekus, S.~Kousidis, and D.~Lohmann,
\emph{Deformations along subsheaves},
L'Enseignement Math\'ematique (2) \textbf{56} (2010), 287--313.

\bibitem{Kol96}
J.~Koll\'ar,
\emph{Rational Curves on Algebraic Varieties},
Ergebnisse der Mathematik und ihrer Grenzgebiete, 3.~Folge,
vol.~32, Springer-Verlag, Berlin, 1996.

\bibitem{LL24} M.-L. Li, W. Liu, \textit{The limits of K\"ahler manifolds under holomorphic deformations}, 	arXiv: 2406.14076, to appear in Acta Mathematica Sinica, English Series.


\bibitem{LRW25}
M.-L.~Li, S.~Rao, and K.~Wang,
\emph{Deformation of nef adjoint canonical line bundles},
preprint, \arxiv{2510.23967v1}, 2025.


\bibitem{Mag}
J.~I.~Magn\'usson,
\emph{Lectures on Cycle Spaces},
lecture notes based on lectures at Ruhr Universit\"at Bochum,
February 2004; Sections~10--11.

\bibitem{Wi91}
J.~A.~Wi\'sniewski,
\emph{On deformation of nef values},
Duke Mathematical Journal \textbf{64} (1991), no.~2, 325--332.

\bibitem{Wi98}
J.~A.~Wi\'sniewski,
\emph{Cohomological invariants of complex manifolds coming from
extremal rays},
preprint, \arxiv{math/9803010v1}, 1998.

\end{thebibliography}
\end{document}